\documentclass[12pt]{amsart}
\input xy \xyoption{all}
\usepackage{enumerate}
\usepackage{amssymb}
\usepackage{graphicx}

\DeclareMathOperator\T{\Bbb T}

\DeclareMathOperator\C{\mathbb C}
\DeclareMathOperator\Z{\mathbb Z}

\DeclareMathOperator\Gr{\mathrm Gr}
\DeclareMathOperator\Fl{\mathrm Fl}
\DeclareMathOperator\FF{\mathcal F}

\DeclareMathOperator\codim{codim}
\DeclareMathOperator\sym{Sym}
\DeclareMathOperator\sgn{sgn}
\DeclareMathOperator\spa{span}
\DeclareMathOperator\Sing{Sing}
\DeclareMathOperator\R{\mathcal R}
\def\Ia{I^{(a, a+1)}}
\DeclareMathOperator\Id{Id}
\DeclareMathOperator\PP{P}
\def\gl2{\mathfrak{gl}_2}
\DeclareMathOperator\sm{sm}
\DeclareMathOperator\m{m}
\def\ep{\epsilon}

\def\opp{\check}

\newtheorem{fact}{Fact}[section]
\newtheorem{lemma}[fact]{Lemma}
\newtheorem{theorem}[fact]{Theorem}
\newtheorem{definition}[fact]{Definition}
\newtheorem{example}[fact]{Example}
\newtheorem{rremark}[fact]{Remark}\newenvironment{remark}{\begin{rremark} \rm}{\end{rremark}}
\newtheorem{proposition}[fact]{Proposition}
\newtheorem{corollary}[fact]{Corollary}

\title[Classes of conormal bundles of Schubert varieties
and weight functions]
{Cohomology classes of conormal bundles of Schubert varieties
and Yangian weight functions}

\author{R. Rim\'anyi}
\address{Department of Mathematics, University of North Carolina at Chapel Hill, USA}
\email{rimanyi@email.unc.edu}

\author{V. Tarasov}
\address{Department of Mathematical Sciences, Indiana University---Purdue University Indianapolis, USA and St.\,Petersburg Branch of Steklov Mathematical Institute, Russia}
\email{vt@math.iupui.edu, vt@pdmi.ras.ru}

\author{A. Varchenko}
\address{Department of Mathematics, University of North Carolina at Chapel Hill, USA}
\email{anv@email.unc.edu}


\begin{document}

% Sasha's version, May 14
%\begin{abstract}
%We consider the conormal bundle $CS_I$ of a Schubert variety $S_I$ in the cotangent bundle $T^*\!\Gr$ of the Grassmannian $\Gr$ of $k$-planes
%in $\C^n$.
%This conormal bundle has a fundamental class in the equivariant cohomology $H^*_{ T}(T^*\!\!\Gr)$.
%Here
% $T=(\C^*)^n\times \C^*$. The torus $(\C^*)^n$ acts on $T^*\!\Gr$ in the standard way and
%  the last factor $\C^*$ acts by multiplication on fibers of the bundle.
%We express this fundamental class as a sum $Y_I$ of the Yangian $Y(\gl2)$ weight functions $(W_J)_J$.
%We describe a relation of $Y_I$ with the double Schur polynomial $[S_I]$.
%\end{abstract}

\begin{abstract}
We consider the conormal bundle of a Schubert variety $S_I$ in the cotangent bundle $T^*\!\Gr$ of the Grassmannian $\Gr$ of $k$-planes
in $\C^n$. This conormal bundle has a fundamental class ${\kappa_I}$
in the equivariant cohomology $H^*_{\T}(T^*\!\!\Gr)$.
Here  $\T=(\C^*)^n\times \C^*$. The torus $(\C^*)^n$ acts on $T^*\!\Gr$ in the standard
way and the last factor $\C^*$ acts by multiplication on fibers of the bundle. We express
 this fundamental class as a sum $Y_I$ of the Yangian $Y(\gl2)$ weight functions $(W_J)_J$.
 We describe a relation of $Y_I$ with the double Schur polynomial $[S_I]$.

A modified version of the $\kappa_I$ classes, named $\kappa'_I$, satisfy an orthogonality relation with respect to an inner product induced by integration on the non-compact manifold $T^*\!\Gr$. This orthogonality is analogous to the well known orthogonality satisfied by the classes of Schubert varieties with respect to integration on $\Gr$.

 The classes $(\kappa'_I)_I$ form a basis in the suitably localized
 equivariant cohomology $H^*_{\T}(T^*\!\!\Gr)$. This basis depends on the choice of the coordinate
 flag in $\C^n$. We show that the bases corresponding to different coordinate flags are related
 by the Yangian R-matrix.

\end{abstract}

\maketitle

\thispagestyle{empty}
\section{Introduction}

The equivariant cohomology of the cotangent bundle of a Grassmannian has hidden Yangian symmetries,
see for example \cite{Vas1, Vas2, vara,naka, yangian}. In this paper we give another example of that symmetry.

The Yangian symmetry considered in this paper is a special case of Yangian symmetries on Nakajima quiver varieties, namely corresponding to the quiver $A_1$ \cite{naka}. In this special case the set of torus fixed points is finite, making it possible to perform fairly explicit calculations.

We study the conormal bundle of a Schubert variety $S_I$ in the cotangent bundle $T^*\!\Gr$ of the Grassmannian $\Gr$ of $k$-planes
in $\C^n$.
We consider a resolution $\tilde{S}_I$ of $S_I$, which lies in a flag variety $\Fl$. The resolution map is the restriction of the
natural forgetful map $\pi: \Fl \to \Gr$. The conormal bundle $C\!\tilde{S}_I\subset T^*\!\Fl$ of $\tilde{S}_I$
has an equivariant fundamental cohomology class $[C\!\tilde{S}_I]$ in the torus equivariant cohomology ring $H^*_{\T}(T^*\!\Fl)=H_{\T}^*(\Fl)$.
Here
 $\T=(\C^*)^n\times \C^*$. The torus $(\C^*)^n$ acts on $T^*\!\Fl$ in the standard way and
  the last factor $\C^*$ acts by multiplication on fibers of the bundle.
Our main object of study is
$$\kappa_I=\pi_*([C\!\tilde{S}_I]) \in H_{\T}^*(\Gr)=H_{\T}^*(T^*\!\Gr),$$
which we call the {\em equivariant fundamental cohomology class} %(or briefly, the fundamantal class)
of the cotangent bundle of $S_I$.

In Theorem \ref{thm:main} we express $e_h\cdot {\kappa_I}$ as a sum $Y_I$ of the Yangian $Y(\gl2)$ weight functions $(W_J)_J$.
Here $e_h$ is an explicit cohomology class independent of $I$, and $e_h$ is not a zero-divisor.

 In Section \ref{sec:schur} we
  consider the fundamental  class $[S_I]$ of the Schubert variety in the equivariant cohomology $H^*_{(\C^*)^n}(\Gr)$
 and obtain  $[S_I]$  as a suitable leading coefficient of the class ${\kappa_I}$. The classes $[S_I]$ are
key objects in Schubert calculus. They are represented by the double Schur polynomials (in the associated Chern roots).
Proposition \ref{prop:leading} and Corollary \ref{cor:leading}
say that the double Schur polynomials $[S_I]$  are leading coefficients of Yangian weight functions $W_I$, and of their sums $Y_I$.

In Section \ref{sec:orto} we define a modified version $\kappa'_I$ of $\kappa_I$.
 The transition matrix from $(\kappa_I)_I$ to $(\kappa_J')_J$ is an upper triangular matrix with integer coefficients and ones in the diagonal.
 For every $I$ the difference $\kappa'_I-\kappa_I$ is supported on $p^{-1}( S_I - S^o_I)$, where $S_I^o$ is the Schubert cell, and $p$ is the projection of the bundle $p:T^*\!\Gr \to \Gr$.  We show that the $\kappa'_I$ classes satisfy an orthogonality relation with respect to an inner product induced by integration on the non-compact manifold $T^*\!\Gr$. This orthogonality is analogous to the well known orthogonality satisfied by the classes of Schubert varieties with respect to integration on $\Gr$.

 The classes $(\kappa'_I)_I$ form a basis in the suitably localized
 equivariant cohomology $H^*_{\T}(T^*\!\!\Gr)$. This basis depends on the choice of the coordinate
 flag in $\C^n$. In Section \ref{sec R} we show that the bases corresponding to different coordinate flags are related
 by the Yangian R-matrix.

\smallskip

The weight functions were used in \cite{TV2} to describe $q$-hypergeometric
solutions of the qKZ equations associated with
the Yangian $Y(\gl2)$. The $q$-hypergeometric solutions are of the form
\begin{equation}
\label{hg}
I_\gamma(z_1,\dots,z_n) = \sum_J \left(\int_\gamma \Phi(t_1,\dots,t_k,z_1,\dots,z_n) W_J(t_1,\dots,t_k,z_1,\dots,z_n) dt_1\dots dt_k\right) v_J
\end{equation}
where $(v_J)$ is a basis of a vector space, $\gamma$ is an integration cycle parametrizing solutions,\\
$\Phi(t_1,\dots,t_k,z_1,\dots,z_n)$ is a (master) function independent of
the index $J$, the functions
\\
$W_J(t_1,\dots,t_k,z_1,\dots,z_n)$ are the weight functions.
To express the fundamental class ${\kappa_I}$ we replace in the weight functions
the integration variables $t_1,\dots,t_k$ with the Chern roots of the canonical bundle over $\Gr$.

As explained in \cite{TV2}, the $q$-hypergeometric solutions (\ref{hg}) identify the qKZ equations
with a suitable discrete Gauss-Manin connection. In \cite{TV2}, the weight functions were identified with the cohomology
classes of a discretization of a suitable de Rham cohomology group. Our formula for the cohomology class of the conormal bundle of a Schubert
variety in terms of weight functions indicates a connection between that discrete de Rham cohomology group of \cite{TV2}
and the equivariant cohomology
of the Grassmannian.

\medskip
This paper is a part of our project of identification of the equivariant cohomology of partial flag varieties with
Bethe algebras of quantum integrable systems, see \cite{V, RV, RSV, RTVZ, yangian}. On this subject see also for example \cite{BMO, NS}.

\bigskip

\noindent{\bf Acknowledgement.} R.R. is supported in part by NSA grant CON:H98230-10-1-0171. V.T.
is supported in part by NSF grant DMS-0901616. A.V. is supported in part by NSF grant DMS-1101508.
A.V. thanks the Hausdorff Research Institute for Mathematics and the Institut des Hautes \'Etudes Scientifiques
for hospitality. We are grateful for L. Feh\'er for useful discussions.

\section{preliminaries}

\subsection{Blocks} \label{sec:prelim:subsets}
The positive integers $k\leq n$ will be fixed throughout the paper.
\smallskip

In Sections \ref{sec:weight} and \ref{sec:geo} we will parameterize various objects in representation theory and geometry by $k$ element subsets $I$ of $\{1,\ldots,n\}$. We will use some notation on the `blocks' in $I$, as follows. Write $I$ as a disjoint union
$$\begin{matrix}
I & = & I_1 & \cup & I_2 & \cup & \ldots & \cup & I_l \\
& = & \{i_1,\ldots, i_{v(1)}\} & \cup & \{i_{v(1)+1}, \ldots, i_{v(2)}\} & \cup & \ldots & \cup & \{i_{v(l-1)+1},\ldots, i_{v(l)}\}
\end{matrix}$$
of maximal intervals of consecutive numbers, that is, we assume
\begin{itemize}
\item{} $i_1<i_2<\ldots < i_{v(l)}$ ($v(l)=k$),
\item{} $i_{v(c-1)+1},\ldots, i_{v(c)}$ is an interval of consecutive integers for $c=1,\ldots,l$ (we put $v(0)=0$),
\item{} $i_{v(c)} + 1 < i_{v(c)+1}$ for all $c=1,\ldots,l-1$.
\end{itemize}
The subsets $I_1,\ldots,I_l$ will be called the blocks of $I$. The lengths of the blocks will be denoted by $m_c=|I_c|=v(c)-v(c-1)$, $c=1,\ldots,l$.
Let $\m=(m_1,\ldots,m_l)$. For $a\in \{1,\ldots,k\}$ let $i_{\hat{a}}$ be the largest element of the block containing $i_a$. That is, $\hat{a}=v(c)$ for some $c\in \{1,\ldots,l\}$. We set $I! = m_1! \cdots m_l!$.

\smallskip

Let $\ell(I)=\sum_{a=1}^k (i_a-a)$.

\subsection{Symmetrizer operations} \label{sec:prelim:sym}

For a function $f$ of the variables $t_1,\ldots,t_k$ we set
$$\sym_{S_k} f = \sum_{\sigma\in S_k} f(t_{\sigma(1)}, \ldots, t_{\sigma(k)}).$$

If a vector $\m=(m_1,\ldots,m_l)\in \{0,1,2,\ldots\}^l$ is given with $v(c)=\sum_{d=1}^c m_d$, $v(l)=k$, then let $S_{\m}$ be the subgroup of $S_k$ permuting the groups of variables
\begin{equation}\label{eqn:vars}
\{t_{1},\ldots, t_{{v(1)}} \} , \qquad\qquad \{t_{{v(1)+1}},\ldots, t_{{v(2)}}\}, \qquad\qquad \ldots \qquad
\{t_{{v(l-1)+1}},\ldots, t_{{v(l)}}\}
\end{equation}
independently. We set
$$\sym_{S_{\m}} f = \sum_{\sigma\in S_{\m}} f(t_{\sigma(1)}, \ldots, t_{\sigma(k)}).$$
If the function $f(t_1,\ldots,t_k)$ is symmetric in the groups of variables (\ref{eqn:vars}), then we define
$$\sym_{S_k/S_{\m}} f = \frac{1}{ \prod_{c=1}^l m_c!} \sym_{S_k} f.$$
Clearly $\sym_{S_k} f = \sym_{S_k/S_{\m}} ( \sym_{S_{\m}} f )$.

\section{Representation theory: the weight functions} \label{sec:weight}

In this section let $z_1,\ldots, z_n ,h, t_1,\ldots,t_k$ be variables.

\subsection{Weight functions}
Let $I=\{i_1, \ldots, i_k\} \subset \{1,\ldots,n\}$, $i_1<\ldots<i_k$. Following \cite{TV2} we call the function
%\begin{eqnarray*}
%W_I(t_1,\ldots,t_k) &=& \prod_{a=1}^k \prod_{u=1}^n (t_a-z_u) \times \\
%&\times & \sym_{S_k} \left(\left( \prod_{a=1}^k \frac{h}{t_a-z_{i_a}}\prod_{1\leq u < i_a} \frac{t_a-z_u+h}{t_a-z_u}\right) \prod_{1\leq a<b\leq k} %\frac{t_a-t_b+h}{t_a-t_b}\right)
%\end{eqnarray*}
\begin{equation}
W_I(t_1,\ldots,t_k)=h^k \sym_{S_k}\left( \prod_{a=1}^k \left( \prod_{u=1}^{i_a-1} (t_a-z_u+h) \prod_{u=i_a+1}^n(t_a-z_u) \prod_{b=a+1}^k \frac{t_a-t_b+h}{t_a-t_b}\right) \right)
\end{equation}
a weight function. For example
\begin{itemize}
\item{} for $k=1$, $1\leq i\leq n$ we have $W_{\{i\}}=h\prod_{u=1}^{i-1} (t_1-z_u+h) \prod_{u=i+1}^n (t_1-z_u)$
\item{} for $n=4$ we have $$W_{\{1,2\}}=h^2(t_1-z_3)(t_1-z_4)(t_2-z_3)(t_2-z_4)\left( (h+t_1+t_2)(h-z_1-z_2) + 2(t_1t_2+z_1z_2)\right).$$
\end{itemize}

For $I=\{i_1<\ldots<i_k\}$ and $J=\{j_1<\ldots <j_k\}$ we define $I\geq J$ if $i_a\geq j_a$ for all $a$. The following interpolation property of weight functions follows directly from the definition.

\begin{lemma}\label{lem:interpolation}
	We have $W_J|_{t_a=z_{j_a}}\not=0$. If $I\not\geq J$ then $W_I|_{t_a=z_{j_a}}=0$. \qed
\end{lemma}

Let $\R=\C[z_1,\ldots,z_n,h]((z_i-z_j+h)^{-1})_{i,j=1,\ldots,n}$. Note that $i=j$ is allowed, hence $z_1-z_1+h=h$ is invertible in $\R$. Consider the $\R$-submodule $M_{k,n}$ of $\R[t_1,\ldots,t_n]$, spanned by the weight functions $W_I(t_1,\ldots,t_n)$ for all $I\subset \{1,\ldots,n\}$, $|I|=k$.

\begin{lemma} \label{lem:free}
	The module $M_{k,n}$ is a free module with basis $W_I$.
\end{lemma}

\begin{proof}
In an $\R$-linear relation among the $W_I$-functions there is a term $W_J$, such that for all other terms $W_I$ in the relation we have $I\not\geq J$. Then the $t_a=z_{j_a}$ substitution cancels all the terms but $W_J$, see Lemma \ref{lem:interpolation}.
\end{proof}

\subsection{R-matrix} \label{sec:R}

For $1\leq a<k$, $I\subset \{1,\ldots,n\}$, $|I|=k$ let $\Ia$ be the $k$-element subset of $\{1,\ldots,n\}$ in which the roles of $a$ and $a+1$ are replaced. That is
\begin{itemize}
\item{} $a\in I$ if and only if $a+1\in \Ia$, and $a+1\in I$ if and only if $a\in \Ia$,
\item{} for $b\not\in \{a, a+1\}$ we have $b\in \Ia$ if and only if $b\in I$.
\end{itemize}

\begin{proposition}
\label{prop:R}
For $1\leq a<k$ we have
\begin{eqnarray}\label{eqn:Ra}
\frac{h}{z_{a+1}-z_a+h} W_I + \frac{z_{a+1}-z_a}{z_{a+1}-z_a+h} W_{\Ia} & = & W_{I}|_{z_a\leftrightarrow z_{a+1}} \nonumber \\
\frac{z_{a+1}-z_a}{z_{a+1}-z_a+h} W_I + \frac{h}{z_{a+1}-z_a+h} W_{\Ia} & = & W_{\Ia}|_{z_a\leftrightarrow z_{a+1}}
\end{eqnarray}
\end{proposition}

\begin{proof}
If $\Ia=I$ then the statement reduces to the fact that in this case $W_I$ is symmetric in $z_a$ and $z_{a+1}$. Otherwise the statement of the proposition follows from the special case of $k=1$, $n=2$, $a=1$ by simple manipulations of the formulas. This special case, namely,
\begin{equation}
\begin{pmatrix} \cfrac{h}{z_2-z_1+h} & \cfrac{z_2-z_1}{z_2-z_1+h} \\
\cfrac{z_2-z_1}{z_2-z_1+h} & \cfrac{h}{z_2-z_1+h} \end{pmatrix}
\begin{pmatrix} h(t_1-z_2) \\ h(t_1-z_1+h) \end{pmatrix} =
\begin{pmatrix} h(t_1-z_1) \\ h(t_1-z_2+h) \end{pmatrix},
\end{equation}
follows by direct calculation.
\end{proof}

%Observe that the definition of weight functions $W_I$ depends on the order of $z_1, z_2,\ldots, z_n$. One may, however, modify the definition of the weight functions by reordering the $z_u$'s. Proposition~\ref{prop:R} means that the free module $M_{k,n}$ {\sl does not} change if we reorder the $z_u$'s, only its basis. Moreover the change of basis matrix between two orders is governed by the quantum $R$-matrix on the left-hand side of (\ref{eqn:Ra}):
%\begin{equation*}
%R_{a}=\begin{pmatrix}
%\cfrac{z_{a+1}-z_a}{z_{a+1}-z_a+h} & \cfrac{h}{z_{a+1}-z_a+h} \\
%\cfrac{h}{z_{a+1}-z_a+h} & \cfrac{z_{a+1}-z_a}{z_{a+1}-z_a+h}
%\end{pmatrix}=
%\frac{ (z_{a+1}-z_a)\cdot \Id + h \cdot \PP}{z_{a+1}-z_a+h},
%\end{equation*}
%where $\PP$ is the permutation of the components.

\newcommand{\e}[1]{ e_{2,1}^{( #1 )} }

Consider the Lie algebra $\gl2$ with its standard generators $e_{i,j}$ for $i,j\in \{1,2\}$. We denote the basis of its vector representation $\C^2$ by $v_+, v_-$. We have $e_{2,1}v_+=v_-$, $e_{2,1}v_-=0$.

For a subset $I \subset \{1,\ldots,n\}$ define $v_I=v_{i_1}\otimes \ldots \otimes v_{i_n}$ where $i_u=-$ for $u\in I$ and $i_u=+$ for $u\not\in I$. The collection of vectors $v_I$ for all subsets $I$ of $\{1,\ldots,n\}$ form a basis of the vector space $(\C^2)^{\otimes n}$. Hence we have the isomorphism of free $\R$-modules

$$
\begin{matrix}
\bigoplus_{k=0}^n M_{k,n} & \cong & (\C^2)^{\otimes n} \otimes \R \\
W_I & \leftrightarrow & v_I
\end{matrix}.
$$

Proposition \ref{prop:R} means that under this identification the following two natural operations on the two sides are identified.
\begin{itemize}
\item{} Substituting $z_a\leftrightarrow z_{a+1}$ on the left-hand side;
\item{} Acting by the `R-matrix'
$$
\frac{ h\cdot \Id + (z_{a+1}-z_a) \cdot \PP^{(a, a+1)}}{z_{a+1}-z_a+h}
$$
on the right-hand side, where $\PP^{(a, a+1)}$ is the transposition of the $a$th and $a+1$st factors of $(\C^2)^{\otimes n}$.
\end{itemize}

\subsection{$Y$-functions}
\label{sec:Y}

The main object of this section is a certain linear combination of the weight functions.
Consider again the $\gl2$ representation $(\C^2)^{\otimes n}$. Let $\e{j}$ denote the action of $e_{2,1}\in\gl2$ on the $j$'th factor of $(\C^2)^{\otimes n}$. Define
$$\Sigma_j=\e{1}+\ldots+\e{j}.$$
For an $I\subset \{1,\ldots,n\}$ recall the notation on the block-structure of $I$ from Section \ref{sec:prelim:subsets} and set
$$\Sigma_I=\left( \frac{1}{m_1!} \Sigma_{v(1)}^{m_1} \right) \left( \frac{1}{m_2!} \Sigma_{v(2)}^{m_2} \right) \ldots
\left( \frac{1}{m_l!} \Sigma_{v(l)}^{m_l} \right) \ \in\ \text{End} ((\C^2)^{\otimes n}).
$$

\begin{definition} \label{def:Y}
For $I\subset \{1,\ldots,n\}$, let $\Sigma_I (v_{\emptyset}) = \sum_{J} c_{J} v_J$ be
the result of application of the operator $\Sigma_I$ to the vector $v_{\emptyset}$. Define
$$Y_I = \sum_J c_J W_J.$$
\end{definition}

\begin{example} \rm
For $k=1$, $i\leq n$ we have $Y_{\{i\}}=\sum_{j=1}^i W_{\{j\}}$.
%\end{example}
%\begin{example} \rm
\par
\noindent For $k=2$, $n=4$ we have
\begin{eqnarray*}
\Sigma_{\{2,4\}}(v_+ \otimes v_+ \otimes v_+ \otimes v_+) & = & \left(( \e{1} + \e{2} )(\e{1}+\e{2}+\e{3}+\e{4})\right) (v_+ \otimes v_+ \otimes v_+ \otimes v_+) \\
& = & 2v_{\{1,2\}} + v_{\{1,3\}} + v_{\{1,4\}} + v_{\{2,3\}} + v_{\{2,4\}},
\end{eqnarray*}
hence
\begin{itemize}
\item $Y_{\{2,4\}}=2W_{\{1,2\}} + W_{\{1,3\}} + W_{\{1,4\}} + W_{\{2,3\}} + W_{\{2,4\}}.$
\end{itemize}
The other $Y$-functions for $k=2$, $n=4$ are
\begin{itemize}
\item{} $Y_{\{1,2\}}=W_{\{1,2\}}$
\item{} $Y_{\{1,3\}}=W_{\{1,2\}} + W_{\{1,3\}}$
\item{} $Y_{\{1,4\}}=W_{\{1,2\}} + W_{\{1,3\}} + W_{\{1,4\}}$
\item{} $Y_{\{2,3\}}=W_{\{1,2\}} + W_{\{1,3\}} + W_{\{2,3\}}$
%\item{} $Y_{\{2,4\}}=2W_{\{1,2\}} + W_{\{1,3\}} + W_{\{1,4\}} + W_{\{2,3\}} + W_{\{2,4\}}$ (an example of a $W$-coefficient $>1$)
\item{} $Y_{\{3,4\}}=W_{\{1,2\}} + W_{\{1,3\}} + W_{\{1,4\}} +W_{\{2,3\}}+ W_{\{2,4\}} + W_{\{3,4\}}.$
\end{itemize}

\end{example}

\subsection{Formulas for $Y$-functions}

\begin{lemma} \label{lem:Ylemma}
Using the notation of Section \ref{sec:prelim:subsets} we have
\begin{multline}\label{eqn:Ybadform}
Y_I = \frac{1}{I!}\sym_{S_k} \left[ \prod_{a=1}^k \left( \prod_{u=1}^{i_{\hat{a}}}(t_a-z_u+h)\prod_{b=1}^{a-1} \frac{t_a-t_b-h}{t_a-t_b} -
\prod_{u=1}^{i_{\hat{a}}}(t_a-z_u)\prod_{b=1}^{a-1} \frac{t_a-t_b+h}{t_a-t_b} \right) \prod_{u=i_{\hat{a}}+1}^n (t_a-z_u)\right].
\end{multline}
In addition, let
\begin{equation} \label{eqn:N_Idef}
N_I=\left(\prod_{c>d} \prod_{i_a\in I_c} \prod_{i_b\in I_d} (t_a-t_b-h) \right)\cdot \prod_{a=1}^k \left( \prod_{u\leq i_{\hat{a}}} (t_a-z_u+h) \prod_{u>i_{\hat{a}}}(t_a-z_u)\right).
\end{equation}
Then
\begin{equation} \label{eqn:Ygoodform}
Y_I=\sym_{S_k/S_{\m}} \frac{ N_I(t_1,\ldots,t_k) + M_I(t_1,\ldots,t_k) }{\prod_{c>d} \prod_{i_a\in I_c} \prod_{i_b\in I_d} (t_a-t_b) } ,
\end{equation}
where $M$ is a polynomial symmetric in the groups of variables (\ref{eqn:vars}) and $M(z_{\sigma(1)},\ldots,z_{\sigma(k)})=0$ for any permutation $\sigma\in S_k$.
\end{lemma}

\begin{proof}
Formula (\ref{eqn:Ybadform}) is a corollary of Lemma 2.21 in \cite{TV2}.

\smallskip

Now consider the function $F$ in $[ \ \ ]$-brackets in (\ref{eqn:Ybadform}).
Our goal is to write its $S_{\m}$ symmetrization in the form
$$\sym_{S_{\m}} F = I! \cdot \frac{ N_I + M_I }{\prod_{c>d} \prod_{i_a\in I_c}
\prod_{i_b\in I_d} (t_a-t_b) }$$ with $N_I$ and $M_I$ having the required properties.

The function $F$ is a product of $k$ factors $Q_1,\ldots,Q_k$, with each factor being a difference of two rational functions. For example the first such factor is
$$ Q_1=\prod_{u=1}^{i_{\hat{1}}}(t_1-z_u+h)\prod_{u=i_{\hat{1}}+1}^n (t_1-z_u) - \prod_{u=1}^{n}(t_1-z_u),$$
and the last such factor is
$$ Q_k=\prod_{u=1}^{i_{\hat{k}}}(t_k-z_u+h)\prod_{u=i_{\hat{k}}+1}^n(t_k-z_u) \prod_{b=1}^{k-1} \frac{t_k-t_b-h}{t_k-t_b} - \prod_{u=1}^{n}(t_k-z_u)\prod_{b=1}^{k-1} \frac{t_k-t_b+h}{t_k-t_b}.$$
When this product of $k$ factors is distributed, we have $2^k$ terms. Each of these terms are of the form
$$\frac{p(t_1,\ldots,t_k)}{\prod_{1\leq b\leq a\leq k} (t_a-t_b)},$$
where $p$ is a polynomial. Hence, the $S_{\m}$ symmetrization of this term is of the form
$$\frac{q(t_1,\ldots,t_k)}{\prod_{c>d} \prod_{i_a\in I_c} \prod_{i_b\in I_d} (t_a-t_b)}$$
for an appropriate polynomial $q$.

We first claim that $q$ satisfies $q(z_{\sigma(1)},\ldots,z_{\sigma(k)})=0$ for any permutation $\sigma\in S_k$, unless the term we are considering comes from the choice of choosing the {\sl first term from each $Q_1,Q_2,\ldots,Q_k$}. Indeed, the second term of $Q_a$ is divisible by
$\prod_{u=1}^n (t_a-z_u)$. Hence $p$ is divisible by this. Therefore $q$ can be written as a sum of terms, each having a factor
$\prod_{u=1}^n (t_b-z_u)$ (for different $b$'s). This proves our first claim. The sum of the $q$ polynomials corresponding to these $2^k-1$ choices will be $I!\cdot M_I$.

Now consider the product of the first terms of $Q_1, \ldots, Q_k$. It is
$$\prod_{a=1}^k \left( \prod_{u=1}^{i_{\hat{a}}} (t_a-z_u+h) \prod_{u=i_{\hat{a}}+1}^n (t_a-z_u) \right) \prod_{1\leq b\leq a\leq k} \frac{t_a-t_b-h}{t_a-t_b}.$$
Its $S_{\m}$ symmetrization is equal
$$\prod_{a=1}^k \left( \prod_{u=1}^{i_{\hat{a}}} (t_a-z_u+h) \prod_{u=i_{\hat{a}}+1}^n (t_a-z_u) \right) \sym_{S_{\m}} \prod_{1\leq b\leq a\leq k} \frac{t_a-t_b-h}{t_a-t_b}.$$
Observe that
\begin{equation}\label{eqn:sym1}
\sym_{S_{\m}} \prod_{1\leq b\leq a\leq k} \frac{t_a-t_b-h}{t_a-t_b}=\prod_{c>d} \prod_{i_a\in I_c} \prod_{i_b\in I_d} \frac{t_a-t_b-h}{t_a-t_b} \cdot
\sym_{S_{\m}} \prod_{\stackrel{a>b}{i_{\hat{a}}=i_{\hat{b}}}} \frac{t_a-t_b-h}{t_a-t_b}.
\end{equation}
Applying the simple identity
$$\sym_{S_r} \prod_{1\leq b\leq a\leq r} \frac{t_a-t_b-h}{t_a-t_b}=r!$$
for $r=m_1, m_2, \ldots$, we find that (\ref{eqn:sym1}) equals
$$\prod_{c>d}\prod_{i_a\in I_c} \prod_{i_b\in I_d} \frac{t_a-t_b-h}{t_a-t_b} \cdot I!.$$
Hence the $q$ polynomial corresponding to the choice of the first factors from $Q_1,\ldots, Q_k$ is $I!\cdot N_I$. This proves the second part of the Lemma.
\end{proof}

\section{Geometry: conormal bundles of Schubert varieties}\label{sec:geo}

\subsection{The Schubert variety and its resolution}

%Our reference for the material concerning Grassmannians, Schubert varieties and their resolutions is \cite{manivel}.

Let $\epsilon_1,\ldots,\epsilon_n$ be the standard basis of $\C^n$. Consider the Grassmannian $\Gr=\Gr_k\!\C^n$ of $k$ dimensional subspaces of $\C^n$, and the standard flag
\begin{equation}
\C^1\subset \C^2\subset \ldots \subset \C^n,
\end{equation}
where $\C^a=\spa(\epsilon_1,\ldots,\epsilon_a)$. For $I=\{i_1<\ldots <i_k\} \subset \{1,\ldots, n\}$
we define the Schubert variety
\begin{equation*}
S_I=\{W^k \subset \C^n : \dim (W^k \cap \C^{i_a})\geq a\ \text{for}\ a=1,\ldots,k\}\subset \Gr.
\end{equation*}

The dimension of $S_I$ is $\ell(I)=\sum_{a=1}^k (i_a-a)$.
Recall the notations of the block-structure of $I$ from Section \ref{sec:prelim:subsets}. In particular, there are indexes $v(1), \ldots, v(l)$  determined by $I$. The definition of $S_I$ can be rephrased (see e.g. \cite[Sect. 3.2]{manivel}) as
\begin{equation*}
S_I=\{W^k \subset \C^n : \dim (W^k \cap \C^{i_{v(c)}})\geq v(c)\ \text{for}\ c=1,\ldots,l\}\subset \Gr.
\end{equation*}
Consider the partial flag variety $\Fl=\Fl_{{v(1)}, {v(2)},\ldots, {v(l)}}\C^n$ of nested subspaces
$$L_1\subset L_2 \subset \ldots \subset L_l$$
of $\C^n$, where $\dim L_c=v(c)$. Note that $v(1),\ldots,v(l)$ depend on $I$, hence $\Fl$ depends on $I$. The natural forgetful map $\Fl\to \Gr$, $(L_1\subset\ldots\subset L_l) \mapsto L_l$ will be denoted by $\pi$. The Schubert variety
\begin{equation}
\tilde{S}_I=\{ (L_1\subset \ldots \subset L_l) : L_c \subset \C^{i_{v(c)}}\ \text{for}\ c=1,\ldots,l\} \subset \Fl
\end{equation}
is a resolution of $S$ through the map $\pi$ \cite[Section 3.4]{manivel}.

\smallskip

Since $\tilde{S}_I$ is a smooth subvariety of $\Fl$, we can consider its conormal bundle
\begin{equation}\label{eqn:def_conormal}
C\!\tilde{S}_I=\{\alpha\in T_p^*\Fl : p\in \tilde{S}_I, \alpha(T_p\tilde{S}_I)=0\} \subset T^*\!\Fl.
\end{equation}
It is a rank $\dim \Fl - \dim \tilde{S}_I$ subbundle of the cotangent bundle of $\Fl$ restricted to $\tilde{S}_I$, hence $\dim C\!\tilde{S}_I=\dim \Fl$. It is well known that $C\!\tilde{S}_I$ is a Lagrangian subvariety of $T^*\!\Fl$ with its usual symplectic form.

\subsection{Torus equivariant cohomology}
Consider the torus ${\T}=(\C^*)^n\times \C^*$. Let $L_i$ be the tautological line bundle over the $i$'th component of the classifying space $B\T=\left({\mathbb P}^{\infty}\right)^{n}\times {\mathbb P}^{\infty}$. We define $z_i=c_1(L_i)$ for $i=1,\ldots,n$, and $h=c_1(L_{n+1})$. We have
$H^*_{\T}(\text{one-point space})=H^*(B\T)=\C[z_1,\ldots,z_n,h]$. The $\T$ equivariant cohomology ring of any $\T$ space is a module over this polynomial ring.

Consider the action of $\T$ on $\C^n$ given by $(\alpha_1,\ldots,\alpha_n)\cdot (v_1,\ldots,v_n)=(\alpha_1v_1,\ldots,\alpha_nv_n)$.
This action induces
an action of ${\T}$ on $\Gr$ and $\Fl$. We will be concerned with the action of ${\T}$ on the cotangent bundles $T^*\!\Gr$ and $T^*\!\Fl$ defined as follows: the action of the subgroup $(\C^*)^n$ is induced from the action of ${\T}$ on $\Gr$ and $\Fl$, while the extra $\C^*$ factor acts by multiplication in the fiber direction. We have the diagram of maps in ${\T}$-equivariant cohomology
\begin{equation*}
\xymatrix{
H^*_{\T}(\Fl) \ar[d]_{\pi_*} & \cong & H^*_{\T}(T^*\!\Fl) \\
H^*_{\T}(\Gr) & \cong & H^*_{\T}( T^*\!\Gr) }
\end{equation*}
where $\pi_*$ is the push-forward map (also known as Gysin map) in cohomology. The isomorphisms $H^*_{\T}(X)\cong H^*_{\T}(T^*X)$ are induced by the equivariant homotopy equivalences $T^*X\to X$. In notation we will not distinguish between $H^*_{\T}(X)$ and $H^*_{\T}(T^*X)$, in particular $\pi_*$ will denote the map
$H^*_{\T}(T^*\!\Fl) \to H^*_{\T}( T^*\!\Gr)$ as well.

\medskip

There are natural bundles over $\Gr$, whose fibers over the point $W\in \Gr$ are $W, \C^n/W$. Let the Chern roots of these bundles be denoted by
\begin{equation} \label{eqn:t1}
\underbrace{t_1,\ldots, t_k}_{k}, \qquad \underbrace{\tilde{t}_1,\ldots,\tilde{t}_{n-k}}_{n-k}.
\end{equation}
Let the group $S_k$ act by permuting the $t_a$'s, and $S_{n-k}$ by permuting the $\tilde{t}_b$'s. We have
$$H^*_{\T}(\Gr)=\C[t_1,\ldots,t_k,\tilde{t}_1,\ldots,\tilde{t}_{n-k},z_1,\ldots,z_n,h]^{S_k\times S_{n-k}}/I_{\Gr},$$
where the ideal $I_{\Gr}$ is generated by the coefficients of the following polynomial in $\xi$:
$$\prod_{a=1}^k (1+t_a \xi) \prod_{b=1}^{n-k} (1+\tilde{t}_b\xi) - \prod_{u=1}^n (1+z_u \xi).$$

There are natural bundles over $\Fl$, whose fibers over the point $(L_1,L_2,\ldots,L_l)\in \Fl$ are $L_1$, $L_2/L_1$, $L_3/L_2, \ldots, L_l/L_{l-1}$, $\C^n/L_l$. Let the Chern roots of these bundles be denoted by
\begin{equation} \label{eqn:t}
\underbrace{t_1,\ldots, t_{v(1)}}_{m_1}, \qquad \underbrace{t_{v(1)+1},\ldots,t_{v(2)}}_{m_2}, \qquad \ldots \qquad \underbrace{t_{v(l-1)+1},\ldots,t_{v(l)}}_{m_l}, \qquad \underbrace{\tilde{t}_1,\ldots,\tilde{t}_{n-k}}_{n-k}.
\end{equation}

Recall that $S_{\m}$ permutes the variables $t_a$ as in Section \ref{sec:prelim:sym}, $S_{n-k}$ permutes the variables $\tilde{t}_b$. We have
\begin{equation} \label{eqn:H_of_Fl}
H^*_{\T}(\Fl)=\C[t_1,\ldots,t_k,\tilde{t}_1,\ldots,\tilde{t}_{n-k},z_1,\ldots,z_n,h]^{S_{\m}\times S_{n-k}}/I_{\Fl},
\end{equation}
where the ideal $I_{\Fl}$ is generated by the coefficients of the following polynomial in $\xi$:
$$\prod_{a=1}^k (1+t_a \xi) \prod_{b=1}^{n-k} (1+\tilde{t}_b\xi) - \prod_{u=1}^n (1+z_u \xi).$$

\smallskip

%Observe that we denoted cohomology classes both on $\Gr$ and $\Fl$ by the letter $t_a$ and $\tilde{t}_b$. This slight abuse of language is justified by the fact that
%$$\text{if} \quad p \in \C[t_1,\ldots,t_k,\tilde{t}_1,\ldots,\tilde{t}_{n-k},z_1,\ldots,z_n]^{S_k\times S_{n-k}\times S_n}\quad \text{then} \quad \pi^*([p])=[p].$$

\subsection{Equivariant localization} \label{sec:localization}
We recall some facts from the theory of equivariant localization, specified for $\Fl$ and $\Gr$.
Our reference is \cite{AB}, a more recent account is e.g. \cite[Ch. 5]{ginzburg}.
Let $\FF_{\Fl}$ be the set of fixed points of the ${\T}$ action on $\Fl$, and
let $\FF_{\Gr}$ be the set of fixed points of the ${\T}$ action on $\Gr$. The sets $\FF_{\Fl}$, $\FF_{\Gr}$ are finite.
Consider the restriction maps
\begin{equation}\label{diag:rest_vs_pi}
H^*_{\T}(\Fl) \to H^*_{\T}(\FF_{\Fl}) = \bigoplus_{f\in \FF_{\Fl}} H^*_{\T}(f), \qquad
H^*_{\T}(\Gr) \to H^*_{\T}(\FF_{\Gr}) = \bigoplus_{f\in \FF_{\Gr}} H^*_{\T}(f).
\end{equation}
A key fact of the theory of equivariant localization is that these restriction maps are injective.
The explicit form of these restriction maps is as follows.
\begin{itemize}
\item{} A fixed point $f\in \FF_{\Gr}$ is a coordinate $k$-plane in $\C^n$, hence it corresponds to a subset $I\subset \{1,\ldots,n\}$, $|I|=k$. The restriction map $H_{\T}^*(\Gr)\to H^*_{\T}(f)=\C[z]$ to the fixed point corresponding to $I$ is
$$[p(t, \tilde{t},z,h)] \mapsto p(z_I, z_{\bar{I}}, z,h),$$
where $t=(t_1,\ldots,t_k)$, $\tilde{t}=(\tilde{t}_1,\ldots, \tilde{t}_{n-k})$, $z=(z_1,\ldots,z_n)$; and $z_I$ stands for the list of $z_u$'s with $u\in I$; and $z_{\bar{I}}$ stands for the list of $z_u$'s with $u\not\in I$.
\item{} Recall that $\Fl$ depends on $I$, and the block structure of $I$ determines numbers $l, m_c$ as in Section \ref{sec:prelim:subsets}.
A fixed point $f\in \Fl$ corresponds to a decomposition $\{1,\ldots,n\}=K_1 \cup \ldots \cup K_l \cup K_{l+1}$ into disjoint subsets $K_1,\ldots, K_l, K_{l+1}$ with $|K_c|=m_c$ for $c=1,\ldots,l$, and $|K_{l+1}|=n-k$. The restriction map $H_{\T}^*(\Fl)\to H^*_{\T}(f)=\C[z]$ to the fixed point corresponding to this decomposition is
$$[p(t, \tilde{t},z,h)] \mapsto p(z_{K_1}, \ldots, z_{K_l}, z_{K_{l+1}}, z,h).$$
\end{itemize}

We will also need a formula for the push-forward map $\pi_*$.

\begin{proposition} Let $[p]\in H^*_{\T}(\Fl)$, where
$p\in\C[t_1,\ldots,t_k,\tilde{t}_1,\ldots,\tilde{t}_{n-k},z_1,\ldots,z_n]^{S_{\m}\times S_{n-k}}$ (see (\ref{eqn:H_of_Fl})). Then
\begin{equation} \label{eqn:int}
\pi_*\left( [p] \right) = \left[ \sym_{S_k/S_{\m}} \frac{p}
{\prod_{c>d} \prod_{t_a\in I_c} \prod_{t_b \in I_d} (t_a-t_b)} \right].
\end{equation}
\end{proposition}

The expression on the right-hand side is formally a sum of fractions. However, in the sum the denominators cancel: the sum is a polynomial.

\begin{proof} We sketch a proof of this formula well known in localization theory. Let $f_I$ be the fixed point on $\Gr$ corresponding to the subset $I=\{i_1,\ldots,i_k\}\subset \{1,\ldots,n\}$. We have
\begin{equation}\label{eqn:last}
\pi_*([p])|_{f_I} = \left(\pi|_{\pi^{-1}(f_I)} \right)_* \left([p]_{\pi^{-1}(f_I)}\right)=
\left(\pi|_{\pi^{-1}(f_I)} \right)_*\left( \sum_{f\in \pi^{-1}(f_I) \cap \FF_{\Fl}} \frac{j_{f*} ( [p]_{f} )}{e(T_{f}\pi^{-1}(f))}\right).
\end{equation}
Here $j_{f}$ is the inclusion $\{f\} \subset \pi^{-1}(f_I)$. In the last equality we used the formula \cite[p.9.]{AB} which describes how one can recover an equivariant cohomology class from its restrictions to the fixed points. Using the fact that $(\pi|_{\pi^{-1}(f_I)} \circ j_{f})_*:\Z[z,h]\to \Z[z,h]$ is the identity map, we obtain that (\ref{eqn:last}) is equal to
$$\left.\left[ \sym_{S_k/S_{\m}} \frac{p}{\prod_{c>d} \prod_{t_a\in I_c} \prod_{t_b \in I_d} (t_a-t_b)} \right]\right|_{t_a=z_{i_a}, \tilde{t}=z_{\bar{I}}.}$$
%$$\sym_{S_k/S_{\m}} \frac{p(z_I, z_{\bar{I}})}{\prod_{c>d} \prod_{t_a\in I_c} \prod_{t_b \in I_d} (z_{i_a}-z_{i_b})}.$$
The restriction of the right-hand-side of (\ref{eqn:int}) to $f_I$ is clearly the same formula. This proves the proposition.
\end{proof}

\subsection{Equivariant fundamental class of the conormal bundle on $\Fl$}\label{sec:fundamental}

The variety $C\!\tilde{S}_I$ is invariant under the ${\T}$
	action on $T^*\!\Fl$. Hence it has an equivariant
	fundamental cohomology class,
	an element $[C\!\tilde{S}_I]\in H_{\T}^{2\dim \Fl}(T^*\!\Fl)=H_{\T}^{2\dim \Fl}(\Fl)$.
%The variety $C\!\tilde{S}_I$ is invariant under the ${\T}$ action on $T^*\!\Fl$,
%hence it represents an equivariant fundamental cohomology class $[C\!\tilde{S}_I]\in H_{\T}^{2\dim \Fl}(T^*\!\Fl)=H_{\T}^{2\dim
% \Fl}(\Fl)$.
Our goal in this section is to express the geometrically defined class $[C\!\tilde{S}_I]$ in terms of the Chern roots $z_u$, $h$, and $t_a$.

Let $\tilde{\gamma}$ be the bundle over $\Fl$ whose fiber over $(L_1,\ldots,L_l)$ is $L_l$.
Consider the ${\T}$ action on the bundle $\tilde{\gamma}^*\otimes \tilde{\gamma}$ where the $(\C^*)^n$ action is induced by the action on $\Fl$, and the extra $\C^*$ acts by multiplication in the fiber direction. The Euler class of this bundle will be denoted by $e_h(\tilde{\gamma}^* \otimes \tilde{\gamma})$. We have
$$e_h(\tilde{\gamma}^* \otimes \tilde{\gamma})=\prod_{a=1}^k \prod_{b=1}^k (t_a-t_b+h).$$

\begin{theorem}
\label{thm:NI_geometry}

For $I\subset \{1,\ldots,n\}$, $|I|=k$ we use the notations of Section \ref{sec:prelim:subsets}.
Let
$$\sgn(I)=(-1)^{\codim S_I}=(-1)^{k(n-k)-\ell(I)}=(-1)^{\sum_{a=1}^k (n-i_{\hat{a}})-\sum_{1\leq c< d \leq l}  m_cm_d}.$$
Consider the cohomology class
\begin
{equation}\label{eqn:N_again}
N_I=\left(\prod_{c>d} \prod_{i_a\in I_c} \prod_{i_b\in I_d} (t_a-t_b-h) \right) \cdot
\prod_{a=1}^k \left( \prod_{u=1}^{i_{\hat{a}}} (t_a-z_u+h) \prod_{u=i_{\hat{a}}+1}^{n} (t_a-z_u) \right) \in H_{\T}^*(\Fl).
\end{equation}
Then
\begin{equation} \label{eqn:NI}
\sgn(I)\cdot e_h(\tilde{\gamma}^* \otimes \tilde{\gamma}) \cdot [C\!\tilde{S}_I] =N_I \qquad\qquad \in H_{\T}^*(\Fl).
\end{equation}
\end{theorem}

Note that expressions (\ref{eqn:N_Idef}) and (\ref{eqn:N_again}) coincide. However, in Lemma \ref{lem:Ylemma},\ $N_I$ is a polynomial, while in Theorem \ref{thm:NI_geometry},\ $N_I$ is a cohomology class (that is, an element of the quotient ring (\ref{eqn:H_of_Fl})).
%Observe that $N_I$ is symmetric in the groups of the $t_a$ variables shown in (\ref{eqn:t}), therefore it is indeed a cohomology class in $H_{\T}^*(\Fl)$.

Note also, that the element $e_h(\tilde{\gamma}^* \otimes \tilde{\gamma})$ is not a zero-divisor in the ring $H_{\T}^*(\Fl)$ (because none of its fixed point restrictions vanish), hence equation (\ref{eqn:NI}) uniquely determines $[C\!\tilde{S}_I]$.

\begin{proof} We will show that the restrictions of the two sides of (\ref{eqn:NI}) to the torus fixed points on $\Fl$ agree.

Let us pick a torus fixed point $f\in \Fl$. It corresponds to a decomposition
$$\{1,\ldots,n\}=K_1\cup\ldots \cup K_l \cup K_{l+1}$$
into disjoint subsets with $|K_c|=m_c$ for $c=1,\ldots,l$ and $|K_{l+1}|=n-k$.
First listing the elements of $K_1$ (in any order), then the elements of $K_2$ (in any order), etc, lastly the elements of $K_l$ (in any order), we obtain a list of numbers $j_1,\ldots,j_k$. Restricting $[p(t_1,\ldots,t_k,z,h)]\in H_{\T}^*(\Gr)$ to the fixed point $f$ amounts to substituting $t_a=z_{j_a}$ into $p$.

\smallskip

%Recall that the standard basis vectors of $\C^n$ are $\ep_1,\ldots, \ep_n$. Recall the block structure of $I$ from Section \ref{sec:prelim:subsets}; in particular, the numbers $m_c, v(c)$ ($c=1,\ldots,l$) associated with $I$; and note that $i_1<\ldots<i_k$. Let $J=(j_1,\ldots,j_k)\in \\N^k$. We do {\em not} assume that the $j_a$'s are ordered. Let $f_J\in \Fl$ be the flag
%$$\spa\{\ep_{j_1},\ldots, \ep_{j_{v(1)}}\} \subset \spa\{ \ep_{j_1}, \ldots, \ep_{j_{v(2)}}\} \subset \ldots \subset \spa\{ \ep_{j_1}, \ldots, \ep_{j_{v(l)}}\} \subset \C^n.$$
%The ${\T}$-fixed points of $\Fl$ are exactly the points $f_J$. (One fixed point corresponds to $I!$ different $J$'s, but what is important for us, is that the $f_J$'s exhaust the list of all fixed points.)

%Since the localization map $H_{\T}^*(\Fl)\to \oplus_J H_{\T}^*(f_J)$ is injective, it is enough to show that the two sides of (\ref{eqn:NI}) have the same restrictions at each fixed point.

Observe that $N_I|_{t_a=z_{j_a}}$ is 0 unless $j_a \leq i_{\hat{a}}$ for all $a$. If $j_a \leq i_{\hat{a}}$ for all $a$ then
\begin{equation*}
N_I|_{t_a=z_{j_a}} = \left(\prod_{c>d} \prod_{i_a\in I_c} \prod_{i_b\in I_d} (z_{j_a} - z_{j_b} -h) \right) \cdot
\prod_{a=1}^k \left( \prod_{u=1}^{i_{\hat{a}}} (z_{j_a} - z_u +h) \prod_{u=i_{\hat{a}}+1}^n (z_{j_a}-z_u)\right) \nonumber
\end{equation*}
\begin{equation*}
= (-1)^{\sum_{c<d} m_cm_d} \left(\prod_{c>d} \prod_{i_a\in I_c} \prod_{i_b\in I_d} (z_{j_b} - z_{j_a} +h) \right) \cdot
\prod_{a=1}^k \left( \prod_{u=1}^{i_{\hat{a}}} (z_{j_a} - z_u +h) \prod_{u=i_{\hat{a}}+1}^n (z_{j_a}-z_u)\right) \nonumber
\end{equation*}
\begin{equation*}
=(-1)^{\sum_{c<d} m_cm_d}
\left(\prod_{a=1}^k \prod_{b=1}^k (z_{j_a}-z_{j_b}+h) \right) \cdot \prod_{a=1}^k
\left( \prod_{\substack{ u\in \{1,\ldots,i_{\hat{a}}\} \\ \ \ - \{ j_1,\ldots, j_{\hat{a}}\} }} (z_{j_a}-z_u +h) \prod_{u=i_{\hat{a}}+1}^n (z_{j_a}-z_u)\right).
\end{equation*}

Now we turn to the study of $[C\!\tilde{S}_I]|_{f}$. %, the cohomology class $[C\!\tilde{S}_I]$ restricted to the ${\T}$-fixed point $f$.

The point $f$ is not contained in $\tilde{S}_I$ unless $j_a \leq i_{\hat{a}}$ for all $a=1,\ldots, k$, hence in this case $[C\!\tilde{S}_I]_{f}=0$. Let us now assume $j_a \leq i_{\hat{a}}$ for all $a=1,\ldots, k$.
The weights of the action of ${\T}$ on the tangent plane $T_f\!\Fl$ are
$$z_{u}-z_{j_a} \qquad \qquad \text{for}\ a=1,\ldots,k, \ \ u\in \{1,\ldots,n\} - \{j_1,\ldots, j_{\hat{a}}\}.$$
Among these weights the ones that correspond to eigenvectors in the tangent space to $\tilde{S}_I$ are
$$z_{u}-z_{j_a} \qquad \qquad \text{for}\ a=1,\ldots,k, \ \ u\in \{1,\ldots,i_{\hat{a}}\} - \{j_1,\ldots, j_{\hat{a}}\}.$$
Hence the weights of the normal space to $\tilde{S}_I$ at $f$ are $z_u-z_{j_a}$ for $u>i_{\hat{a}}$. The weights of the normal space to $C\!\tilde{S}|_{f} \subset T^*\!\Fl|_{f}$ are hence $z_{j_a}-z_u +h$ for $u\in \{1,\ldots,i_{\hat{a}}\}- \{ j_1,\ldots, j_{\hat{a}}\} $.

Therefore we have
$$[C\!\tilde{S}_I]|_{f} = \prod_{a=1}^k \left( \prod_{\substack{ u\in \{1,\ldots,i_{\hat{a}}\} \\ \ \ - \{ j_1,\ldots, j_{\hat{a}}\} }} (z_{j_a}-z_u +h) \prod_{u=i_{\hat{a}}+1}^n (z_u-z_{j_a}) \right).$$

Since the restriction of $e_h(\tilde{\gamma}^* \otimes \tilde{\gamma})$ to the fixed point $f$ is $\prod_{a=1}^k \prod_{b=1}^k (z_{j_a}-z_{j_b}+h)$, we proved that the two sides of (\ref{eqn:NI}) are equal restricted to the fixed point $f$.
\end{proof}

\subsection{Equivariant fundamental class of the conormal bundle on $\Gr$} \label{sec:fundamental_on_Gr}

\begin{definition}\label{def:CSonGr}
We define the equivariant fundamental cohomology class of the conormal bundle of $S_I \subset \Gr$ to be $\pi_*([C\!\tilde{S}_I])$, and we will denote this class by ${\kappa_I}$.
\end{definition}

%The variety $S_I\subset \Gr$ is not smooth, hence the definition of its conormal bundle is not straightforward (we will discuss %this in more detail below). Nevertheless we define its fundamental cohomology class.

%We first prove an expression for this class and then discuss its properties.
%\smallskip

Let $\gamma$ be the bundle over $\Gr$ whose fiber over $W\subset \C^n$ is $W$.
Consider the ${\T}$ action on the bundle $\gamma^*\otimes \gamma$ where the $(\C^*)^n$ action is induced by the action on $\Gr$, and the extra $\C^*$ acts by multiplication in the fiber direction. The Euler class of this bundle will be denoted by $e_h({\gamma}^* \otimes {\gamma})$. We have
$$e_h(\gamma^* \otimes \gamma)=\prod_{a=1}^k \prod_{b=1}^k (t_a-t_b+h).$$

Recall the definition of the $Y_I$-function from Section \ref{sec:Y}. Interpret the variables $t$, $z$, and $h$ of this function according to the definitions of the present section, that is, as equivariant cohomology classes in~$\Gr$. Then $Y_I\in H_{\T}^*(\Gr)$. Recall that $\sgn(I)=(-1)^{\codim S_I}$.

\begin{theorem} \label{thm:main}
%Recall the definition of $\sgn(I)$ and $e_h(\gamma^* \otimes \gamma)$ from Section \ref{sec:fundamental}
We have
\begin{equation*}
\sgn(I) \cdot e_h(\gamma^* \otimes \gamma) \cdot {\kappa_I} =Y_I \qquad\qquad \in H_{\T}^*(\Gr).
\end{equation*}
\end{theorem}

Note that $e_h(\gamma^* \otimes \gamma)$ is not a zero-divisor (because none of its fixed point restrictions vanish), hence the equation in Theorem~\ref{thm:main} uniquely determines ${\kappa_I}$.

\begin{proof}
Let us apply $\pi_*$ to (\ref{eqn:NI}). The left hand side will map to $\sgn(I) e_h(\gamma^*\otimes \gamma) {\kappa_I}$, because the bundle $\tilde{\gamma}$ over $\Fl$ is the pullback of the bundle $\gamma$ over $\Gr$.

Recall the definition of $M_I$ from (\ref{eqn:Ygoodform}). As a cohomology class in $H^*_{T}(\Fl)$, $M_I$ is zero, because it restricts to 0 at every ${\T}$ fixed points of $\Fl$. Hence, $\pi_*$ applied to the right hand side of (\ref{eqn:NI}) is
$$\pi_*(N_I) = \pi_*(N_I + M_I) = Y_I.$$
The last equality holds because of the comparison of (\ref{eqn:Ygoodform}) and (\ref{eqn:int}).
\end{proof}

\medskip

%In the rest of this section we will discuss to what extent is the definition ${\kappa_I}$ a natural one for the %equivariant fundamental cohomology class of the conormal bundle of the singular variety $S_I \subset \Gr$.

Let $p: T^*\!\Gr \to \Gr$ be the projection of the bundle. Let $\Sing S_I\subset S_I$ be the subvariety of singular points of~$S_I$ and
$S^{\sm}_I=S_I-\Sing S_I$ the smooth part. Define the conormal bundle of the smooth part to be
\begin{equation}
\label{eqn:def_conormal_S MY}
C\!S^{\sm}_I=\{\alpha\in T_x^*\Gr : x\in S^{\sm}_I, \alpha(T_xS_I)=0\} \ \ \subset p^{-1}(\Gr - \Sing S_I).
\end{equation}
%Denote $i: C\!S^{\sm}_I {\hookrightarrow} p^{-1}(\Gr - \Sing S_I)$ the natural embedding.
Denote
$
[C\!S^{\sm}_I] \in H^*_{\T}(p^{-1}(\Gr - \Sing S_I))
$
the equivariant fundamental cohomology class of $C\!S^{\sm}_I$. Consider the embedding
$j:p^{-1}(\Gr - \Sing S_I) \hookrightarrow T^*\!\Gr.$

%We define the conormal bundle of $S^{\sm}_I=S_I-\Sing S_I$
%\begin{equation}\label{eqn:def_conormal_S}
%C\!S^{\sm}_I=\{\alpha\in T_x^*\Gr : x\in S_I - \Sing S_I, \alpha(T_xS_I)=0\} \stackrel{i}{\hookrightarrow} p^{-1}(\Gr - \Sing S_I)
%%\left(T^*\!\Gr\right)_{\Gr - \Sing S_I}.
%\end{equation}
%Here, $\sm$ refers to {\em smooth}, and note that we denoted the embedding by $i$.

\begin{theorem} \label{thm:on_gr}
We have $
j^*({\kappa_I}) = [C\!S^{\sm}_I]
\in H^*_{\T}(p^{-1}(\Gr - \Sing S_I)).$
\end{theorem}

%\begin{remark} \rm
%The most standard definition for the conormal bundle $C_1S_I$ of $S_I$ is the Zariski closure of $C\!S^{\sm}_I$ is $T^*\!\Gr$. An %alternative definition for the conormal bundle could be
%$$C_2S_I= C_1S_I \cup C_1(\Sing S_I) \cup C_1(\Sing (\Sing S_I)) \cup \ldots$$
%Both of these definitions (and any other conceivable definition should) agree within $p^{-1}(\Gr$ $- \Sing S_I)$. That is, the %fundamental cohomology classes of the alternative definitions all satisfy that their $j^*$-image is $[C\!S^{\sm}_I]$. Theorem %\ref{thm:on_gr} claims that our definition of the fundamental class of the conormal bundle, ${\kappa_I}$, also has %this property.
%\end{remark}

\begin{proof} To prove Theorem \ref{thm:on_gr} first we recall two lemmas from Schubert calculus that are probably known to the specialists, but we sketch their proofs because we did not find exact references.

\begin{lemma} \label{lem:one_fixed_point}
Let $f_J$ be a smooth ${\T}$ fixed point on $S_I$. Then there is exactly one ${\T}$ fixed point in $\tilde{S}_I \cap\pi^{-1}(f_J)$.
\end{lemma}

\begin{proof} Recall the notation of the blocks in $I$, and assume $J=\{j_1<\ldots<j_k\}$. Assume that $f_J\in S_I$, hence $j_a\leq i_a$ for all $a=1,\ldots,k$. The description of the components of the singular locus of $S_I$ in \cite[Thm.~3.4.4]{manivel} can be rephrased to our language as follows: The ${\T}$ fixed point $f_J$ is a singular point on $S_I$ if and only if $j_{v(c)+1}\leq i_{v(c)}$ for some $c$. If this does not happen for any $c$ then
$$|\{a: j_a\leq v(c)\}|=v(c)$$
for all $c=1,\ldots,l$. In this case there is only one choice for a fixed point in $\tilde{S}_I \cap \pi^{-1}(f_J)$, namely
$$\spa\{\ep_{j_1},\ldots, \ep_{j_{v(1)}}\} \subset \spa\{ \ep_{j_1}, \ldots, \ep_{j_{v(2)}}\} \subset \ldots \subset \spa\{ \ep_{j_1}, \ldots, \ep_{j_{v(l)}}\}.$$
\end{proof}

\begin{lemma} \label{lem:MV}
The localization map
$$H_{\T}^*(\Gr - \Sing S_I) \to \bigoplus_{f} H_{\T}^*(f), \qquad\qquad
\alpha\mapsto \left( \alpha|_{f} \right),$$
where $\oplus$ runs for the ${\T}$ fixed points $f$ in $\Gr - \Sing S_I$, is injective.
\end{lemma}

\begin{proof}
The injectivity of the localization map is usually phrased for compact manifolds (see for example the original \cite{AB}). However, here we sketch an argument for $\Gr- \Sing S_I$. Starting with the open Schubert cell, we add the cells of $\Gr-\Sing S_I$ one by one, in order of codimension. At each step we can use a Mayer-Vietoris argument to show that if the localization map was injective before adding the cell, then it is injective after adding the cell. A strictly analogous argument is shown in the proof of \cite[Lemma~5.3]{cr}. The only condition for such a step-by-step Mayer-Vietoris argument to work is that the equivariant Euler-class of the normal bundle of each cell is not a zero-divisor. This calculation is done for Grassmannians, for example, in \cite[Sect.~5]{ss}.
\end{proof}

Now we can prove Theorem \ref{thm:on_gr}. Let $f_J$ be a smooth ${\T}$ fixed point of $S_I$, and $\tilde{f}_J\in \Fl$ the unique ${\T}$ fixed point in $\tilde{S}_I\cap \pi^{-1}(f_J) $.
An analysis of the ${\T}$ representations on $T_{f_J}\!\Gr$ and $T_{\tilde{f}_J}\!\Fl$, similar to the one in the proof of Theorem \ref{thm:NI_geometry}, gives
\begin{equation}
[C\!\tilde{S}_I]|_{\tilde{f}_J}=[C\!S^{\sm}_I]|_{f_J} \cdot \prod_{c>d} \prod_{i_a\in I_c} \prod_{i_b\in I_d} (z_{i_a}-z_{i_b}).
\end{equation}
Therefore we have
$$\left. {\kappa_I} \right|_{f_J} = \pi_*\left( [C\!\tilde{S}_I]|_{\tilde{f}_J} \right) = [C\!S^{\sm}_I]|_{f_J}.$$
If $f$ is a ${\T}$ fixed point in $\Gr - S_I$ then obviously both $\left. {\kappa_I} \right|_{f}$ and $[C\!S^{\sm}_I]|_{f}$ are 0. Thus we found that
the localization map of Lemma \ref{lem:MV} maps $j^*({\kappa_I}) - [C\!S^{\sm}_I]$ to 0. This proves the theorem.
\end{proof}

\section{Schur polynomials} \label{sec:schur}

\begin{definition} For $I=\{i_1<i_2<\ldots<i_k\} \subset \{1,\ldots,n\}$ we define the double Schur polynomial by
\[
\Delta_I(t_1,\ldots,t_k,z_1,\ldots,z_n)=(-1)^{\codim S_I} \det\left( \prod_{u=i_a+1}^n (t_\beta-z_u) \right)_{\alpha,\beta=1,\ldots,k}
\frac{1}{\prod_{1\leq a<b\leq k} (t_a-t_b)}.
\]
\end{definition}

\begin{remark} \rm
This definition of double Schur polynomials is the so-called {\em bialternant} definition. Other definitions include a (generalized) Jacobi-Trudi determinant, an interpolation definition, and the fact that double Schur polynomials are special cases of double Schubert polynomials that are described recursively using divided difference operators. For references see \cite{MD1}, \cite[Ex.20, Section~I.3, p.54]{MD2}, \cite[Lecture 8]{Fulton}, or \cite{ss}.
\end{remark}

In Schubert calculus it is well known that equivariant classes of Schubert varieties in Grassmannians are represented by double Schur polynomials.

\begin{theorem} (See \cite[Lecture 8]{Fulton} and references therein.) \label{them:S_I_is Schur}
The $(\C^*)^n$-equivariant fundamental class $[S_I]$ of $S_I$ in $H^*_{(\C^*)^n}(\Gr)$ is represented by the double Schur polynomial $\Delta_I(t,z)$.
\end{theorem}

What connects this fact with the objects of the present paper is the following observation.

\begin{proposition} \label{prop:leading} Recall that $\ell(I)=\sum_{a=1}^k (i_a-a)$ is the dimension of $S_I$.
Consider the weight function $W_I$ as a polynomial in $h$. We have
$$W_I=\sgn(I) \Delta_I(t,z) \cdot h^{k^2+\ell(I)} + \text{lower degree terms}.$$
\end{proposition}

\begin{proof}
The statement follows from the explicit formulae for $W_I$ and $\Delta_I$.
\end{proof}

\begin{corollary} \label{cor:leading}
We have
\[
\begin{array}{rcrll}
 Y_I & =& \sgn(I) \Delta_I(t,z) &  h^{k^2+\ell(I)} & + \text{lower degree terms}. \\
 {\kappa_I}& = & [S_I] & h^{\ell(I)}   &  + \text{lower degree terms}.
\end{array}
\]
\end{corollary}

\begin{proof} The first statement follows from the fact $Y_I=W_I+$ sums of $W_J$-functions with $J < I$. The $h$-degrees of these $W_J$ functions are strictly less than the $h$-degree of $W_I$. The second statement is the consequence of Theorems \ref{thm:main} and \ref{them:S_I_is Schur}. Note that the second statement could also be proved geometrically.
\end{proof}

\section{Modified equivariant class of the conormal bundle}

Recall that
\begin{equation*}
Y_I=W_I+\sum_{J<I} c_{I\!J} W_J,
\end{equation*}
where $c_{IJ}$ are some positive integer coefficients. In other words, the transition matrix from the $W_I$ functions to the $Y_I$ functions is triangular (with 1's in the diagonal). Hence the same is true for its inverse:
\begin{equation*}
W_I=Y_I+\sum_{J<I} c'_{I\!J} Y_J,
\end{equation*}
where $c'_{IJ}$ are some integer coefficients. Comparing this expression with Theorem \ref{thm:main} we obtain that in $H_{\T}^*(\Gr)$ we have
\begin{equation*}
\begin{aligned}
W_I & =  \sgn(I) e_h(\gamma^* \otimes \gamma) {\kappa_I} +\sum_{J<I} c'_{I\!J} \sgn(J) e_h(\gamma^*\otimes \gamma) {\kappa_J} \\
    & =  \sgn(I) e_h(\gamma^* \otimes \gamma) \left( {\kappa_I}+ \sum_{J<I} c'_{I\!J} \frac{\sgn(J)}{\sgn(I)} {\kappa_J} \right).
\end{aligned}
\end{equation*}

\begin{definition} \label{def:modified}
We define the {\em modified equivariant fundamental class} of the conormal bundle of the Schubert variety $S_I$ to be
$${\kappa'_I}={\kappa_I}+ \sum_{J<I} c'_{I\!J} \frac{\sgn(J)}{\sgn(I)} {\kappa_J}.$$
\end{definition}

Notice that $J<I$ if and only if $S_J$ is a proper subvariety of $S_I$ (hence, in particular, $\dim S_J<\dim S_I$). Therefore the difference of $\kappa'_I$ and $\kappa_I$ is a linear combination of fundamental classes of conormal bundles of proper subvarieties of $S_I$.

\begin{example}\rm
For $k=1$ we have $Y_{\{i\}}=\sum_{j\leq i} W_{\{j\}}$. Hence for $i>1$
\begin{equation*}
\begin{aligned}
W_{\{i\}} & = Y_{\{i\}}- Y_{\{i-1\}} \\
& = (-1)^{n-i} e_h(\gamma^*\otimes \gamma) \kappa_{\{i\}} - (-1)^{n-(i-1)} e_h(\gamma^*\otimes \gamma) \kappa_{\{i-1\}} \\
& = (-1)^{n-i} e_h(\gamma^*\otimes \gamma) (\kappa_{\{i\}} + \kappa_{\{i-1\}}).
\end{aligned}
\end{equation*}
Therefore $\kappa'_{\{i\}}=\kappa_{\{i\}}+\kappa_{\{i-1\}}=[C\!S_{\{i\}}] + [C\!S_{\{i-1\}}]$.

For example, for $k=1$, $n=2$, let $p:T^*\!\Gr \to \Gr$ be the projection of the bundle as before,
 and $f_{\{1\}}$ the $\T$ fixed point with homogeneous coordinate $(1:0)$ on $\Gr=\Gr_1\!\C^2={\mathbb P}^1$. Then
\begin{align*}
\kappa'_{\{1\}} & =\kappa_{\{1\}}=[p^{-1}(f_{\{1\}})]=z_2-t_1,\\
\kappa'_{\{2\}} & =\kappa_{\{2\}}+\kappa_{\{1\}}=[\Gr] + [p^{-1}(f_{\{1\}})] = (2t_1-z_1-z_2+h)+(z_2-t_1)=t_1-z_1+h.
\end{align*}

\smallskip
For $k=2, n=4$ calculation shows
$$
\kappa'_{\{1,2\}}=\kappa_{\{1,2\}}, \qquad
\kappa'_{\{1,3\}}=\kappa_{\{1,3\}}+\kappa_{\{1,2\}}, \qquad
\kappa'_{\{1,4\}}=\kappa_{\{1,4\}}+\kappa_{\{1,3\}}, \qquad
\kappa'_{\{2,3\}}=\kappa_{\{2,3\}}+\kappa_{\{1,3\}}, $$
$$
\kappa'_{\{2,4\}}=\kappa_{\{2,4\}}+\kappa_{\{2,3\}}+\kappa_{\{1,4\}}+\kappa_{\{1,3\}}+\kappa_{\{1,2\}}, \qquad
\kappa'_{\{3,4\}}=\kappa_{\{3,4\}}+\kappa_{\{2,4\}}+\kappa_{\{1,2\}}.
$$
\end{example}

Let us compare the modified fundamental class with the earlier, non-modified ${\kappa_I}$ classes. We have
\begin{equation}\label{eqn:W1}
\begin{matrix}
\sgn(I) \  e_h(\gamma^* \otimes \gamma) \   {\kappa_I} & = & Y_I, \\
\sgn(I) \  e_h(\gamma^* \otimes \gamma) \   {\kappa'_I} & = & W_I.
\end{matrix}
\end{equation}
Recall from Theorem \ref{thm:on_gr} that $j^*({\kappa_I}) = [C\!S^{\sm}_I] \in H^*_{\T}(p^{-1}(\Gr - \Sing S_I))$.
The class $\kappa'_I$ does not satisfy this property over the whole $\Gr -\Sing S_I$, but only over the Schubert cell
$S^o_I$ which is a dense open subset of  $S_I$.

\begin{theorem}
Let $p:T^*\Gr \to \Gr$ be the projection of the bundle. Consider the Schubert cell
$$S^o_I=\{W^k \subset \C^n : \dim (W^k \cap \C^{i_a})= a\ \text{for}\ a=1,\ldots,k\}\subset \Gr,$$
and its conormal bundle $C\!S^o_I\subset p^{-1}(\Gr - \cup_{J<I} S_J)$. Let $i$ denote the embedding
$$i:p^{-1}(\Gr - \cup_{J<I} S_J) \hookrightarrow T^*\!\Gr.$$
Then
$$i^*({\kappa'_I})=[C\!S^o_I]\ \in H^*(p^{-1}(\Gr - \cup_{J<I} S_J)).$$
\end{theorem}

\begin{proof}  The difference $\kappa'_I-\kappa_I$ is supported on $p^{-1}(\cup_{J<I} S_J)$. Therefore $i^*(\kappa'_I)=i^*(\kappa_I)$, so it is enough to show that $i^*({\kappa_I})=[C\!S^o_I]$. This follows from the diagram
\begin{equation*}
\xymatrix{p^{-1}(\Gr - \cup_{J<I} S_J) \ar@{^{(}->}[r] \ar@{^{(}->}@/^2pc/[rr]^{i} & p^{-1}(\Gr - \Sing S_I) \ar@{^{(}->}[r]_{\ \ \ \ \ \ \ j} & T^*\Gr \\
i^*({\kappa_I}) & j^*({\kappa_I}) \ar@{|-{>}}[l] \ar@{=}[d]& {\kappa_I}\ar@{|-{>}}[l] \\
[C\!S^o_I] & [C\!S^{\sm}_I]. \ar@{|-{>}}[l] &
}
\end{equation*}
\end{proof}

Some advantages of ${\kappa'_I}$ over ${\kappa_I}$ are discussed in Sections \ref{sec:orto} and \ref{sec R}.

\section{Orthogonality}\label{sec:orto}

\subsection{Orthogonality on $\Gr$}
Recall that $\epsilon_1, \ldots, \epsilon_n$ is the standard basis of $\C^n$, and we used the standard flag to define the Schubert variety
\begin{equation*}
S_I=\{W^k \subset \C^n : \dim (W^k \cap \spa(\epsilon_1,\ldots,\epsilon_{i_a}))\geq a\ \text{for}\ a=1,\ldots,k\}\subset \Gr.
\end{equation*}
Considering the opposite flag, we may define the opposite Schubert variety
\begin{equation*}
\opp{S}_I=\{W^k \subset \C^n : \dim (W^k \cap \spa(\epsilon_{n+1-i_a},\ldots,\epsilon_{n}))\geq a\ \text{for}\ a=1,\ldots,k\}\subset \Gr.
\end{equation*}

Consider the bilinear form on $H_{\T}^*(\Gr)$ defined by
$$(f,g) \mapsto \langle f,g \rangle = \int_{\Gr} fg,$$
where the equivariant integral on $\Gr$ can be expressed via localization by
$$\int_{\Gr} \alpha(t_1,\ldots,t_k) = \sum_I \frac{\alpha(z_I)}{\prod_{u\in I, v\not\in I} (z_v-z_u)}.$$
Here the sum runs for $k$-element subsets $I$ of $\{1,\ldots,n\}$. Note that the denominator is the equivariant Euler class of the tangent space to $\Gr$ at the fixed point corresponding to $I$.

%The intersection $S_I \cap S'_J$ is either empty, a positive dimensional variety, or one point. In the latter case the intersection is transversal. This latter case is achieved for

\begin{definition}
For $I=\{i_1<\ldots<i_k\}$ define $\opp{I}=\{(n+1)-i_k, (n+1)-i_{k-1}, \ldots, (n+1)-i_1\}$.
\end{definition}

It is well known in Schubert calculus that
$$\langle [S_I], [\opp{S}_{\opp{J}}] \rangle = \delta_{I,J}.
%\begin{cases} 1 & \text{if } I=J \\ 0 & \text{otherwise.} \end{cases}
$$

\subsection{Orthogonality on $T^*\!\Gr$}

Our goal is to describe similar orthogonality relations involving equivariant classes of conormal bundles.

Consider the bilinear form on $H_{\T}^*(T^*\!\Gr)=H_{\T}^*(\Gr)$ defined by
$$(f,g) \mapsto \ll f,g \gg = \int_{T^*\!\Gr} fg,$$
where the equivariant integral on $T^*\!\Gr$ is {\em defined} via localization by
$$\int_{T^*\!\Gr} \alpha(t_1,\ldots,t_k) = \sum_I \frac{\alpha(z_I)}{\prod_{u\in I, v\not\in I} (z_v-z_u)(z_u-z_v+h)}.$$
Here again, the sum runs for $k$-element subsets $I$ of $\{1,\ldots,n\}$. Note that the denominator is the equivariant Euler class of the tangent space to $T^*\!\Gr$ at the fixed point corresponding to $I$. This bilinear form takes values in the ring of rational functions in $z_1,\ldots,z_n$ and $h$. For more details, see \cite[Sect. 5.2]{yangian}.

In Section \ref{sec:geo}, starting with the Schubert variety $S_I$ we defined the equivariant fundamental classes ${\kappa_I}$ and ${\kappa'_I}$.
Similarly, starting with the opposite Schubert variety $\opp{S}_I$, we may define the opposite equivariant fundamental classes $\opp{\kappa}_I$ and $\opp{\kappa}_I'$. Arguments analogous to the ones in the sections above prove the following theorem.

\begin{theorem}\label{thm:mainv}
We have
\begin{equation}\label{eqn:W2}
\sgn(\opp{I}) e_h(\gamma^*\otimes \gamma) \opp{\kappa}'_{\opp{I}}=\opp{W}_I,
\end{equation}
where
\begin{equation}
\opp{W}_I(t_1,\ldots,t_k)=h^k \sym_{S_k}\left( \prod_{a=1}^k \left( \prod_{u=1}^{i_a-1} (t_a-z_u) \prod_{u=i_a+1}^n(t_a-z_u+h) \prod_{b=a+1}^k \frac{t_a-t_b-h}{t_a-t_b}\right) \right)
\end{equation}
%\prod_{a=1}^k \prod_{u=1}^n (t_a-z_u+h) \times \\
%&\times & \sym_{S_k} \left(\left( \prod_{a=1}^k \frac{h}{t_a-z_{i_a}+h}\prod_{1\leq u < i_a} \frac{t_a-z_u}{t_a-z_u+h}\right) \prod_{1\leq a<b\leq k} %\frac{t_a-t_b-h}{t_a-t_b}\right).
%\end{eqnarray*}
is the ``dual'' weight function.
\end{theorem}

The following lemma is a special case of Theorem C.9 of \cite{TV1}.

\begin{lemma} \label{lem:orto}
For $A=\{a_1,\ldots,a_k\}\subset \{1,\ldots,n\}$ let $p(z_A)$ denote the substitution of $z_{a_1},\ldots,z_{a_k}$ into the variables $t_1,\ldots, t_k$ of the polynomial $p$. For $k$-element subsets $I,J$ of $\{1,\ldots,n\}$ we have
$$\sum_A \frac{W_I(z_A) \opp{W}_J(z_A)}{\prod_{u\in A,v\in \bar{A}} (z_u-z_v)(z_u-z_v+h) \cdot \prod_{u,v\in A} (z_u-z_v+h)^2} = \delta_{I,J},$$
where the summation runs for all $k$-element subsets $A$ of $\{1,\ldots,n\}$.
\end{lemma}

Now we are ready to prove the orthogonality relations for fundamental classes of conormal bundles.

\begin{theorem} \label{thm:orto}
We have
$$\ll \kappa'_I , \opp{\kappa}'_{\opp{J}} \gg = \delta_{I,J}.
%\begin{cases} 1 & \text{if } I=J \\ 0 & \text{otherwise.} \end{cases}
$$
\end{theorem}

\begin{proof}
Tracing back the definitions, as well as formulas (\ref{eqn:W1}) and (\ref{eqn:W2}), give
$$\ll \kappa'_I , \opp{\kappa}'_{\opp{J}} \gg=\sgn(I)\sgn(\opp{J})
\sum_A \frac{W_I(z_A)\opp{W}_J(z_A)}{\prod_{u\in A, v\not\in A}(z_v-z_u)(z_u-z_v+h)\cdot \left(e_h(\gamma^*\otimes \gamma)(z_A)\right)^2}.$$
According to Lemma \ref{lem:orto} this further equals
$$\sgn(I)\sgn(\opp{J})(-1)^{k(n-k)} \delta_{I,J}=\delta_{I,J}.$$
\end{proof}

\section{R-matrix in cohomology}
\label{sec R}

Recall that $\R=\C[z_1,\ldots,z_n,h]((z_i-z_j+h)^{-1})_{i,j}$.

\begin{proposition} \label{prop:weights_basis} The following are bases of the free module $H_{\T}^*(\Gr)\otimes \R$:
\begin{enumerate}
\item{} \label{q1} the classes $[S_I]$ for $I\subset \{1,\ldots,n\}$, $|I|=k$;
\item{} \label{q2} the classes $W_I$ for $I\subset \{1,\ldots,n\}$, $|I|=k$;
\item{} \label{q3} the classes $Y_I$ for $I\subset \{1,\ldots,n\}$, $|I|=k$;
\item{} \label{q4} the classes $\kappa_I$ for $I\subset \{1,\ldots,n\}$, $|I|=k$;
\item{} \label{q5} the classes $\kappa'_I$ for $I\subset \{1,\ldots,n\}$, $|I|=k$.
\end{enumerate}
\end{proposition}

\begin{proof}
The Leray-Hirsch Theorem (e.g. \cite[Thm. 5.11]{BT}) implies that the $\C[z_1,\ldots,z_n,h]$-module $H_{\T}^*(\Gr)$ is free, with basis $[S_I]$. This implies (\ref{q1}). Statement (\ref{q5}) can be proved from Theorem \ref{thm:orto} and some extra analysis of the denominators, or from the fact that the transition matrix from $[S_I]$ to $\kappa'_I$ is upper triangular with diagonal entries
$$\prod_{a\in I, b\not\in I, b<a} (z_a-z_b+h).$$
Details will be given elsewhere. The equivalence of (\ref{q5}) and (\ref{q2}) follows from equivariant localization and the fact $\pm e_h(\gamma^*\otimes \gamma)\kappa'_I=W_I$. The transition matrix between (\ref{q2}) and (\ref{q3}), as well as the transition matrix between (\ref{q4}) and (\ref{q5}) are upper triangular with 1's in the diagonal.
\end{proof}

Recall that $S_I$ was defined using the standard flag
$$\spa(\epsilon_1) \subset \spa(\epsilon_1,\epsilon_2) \subset \ldots \subset \spa(\epsilon_1,\ldots,\epsilon_n).$$
Let $\sigma\in S_n$ be a permutation. We may, redefine our geometric objects $S_I$, $\tilde{S}_I$, $[C\!\tilde{S}_I]$, ${\kappa_I}$, ${\kappa'_I}$ using the complete flag
$$\spa(\epsilon_{\sigma(1)}) \subset \spa(\epsilon_{\sigma(1)},\epsilon_{\sigma(2)}) \subset \ldots \subset \spa(\epsilon_{\sigma(1)},\ldots,\epsilon_{\sigma(n)}).$$
We call the resulting objects $S^{\sigma}_I$, $\tilde{S}^{\sigma}_I$, $[C\!\tilde{S}^{\sigma}_I]$, $\kappa_I^{\sigma}$, ${\kappa'_I}^{\sigma}$.

%Proposition \ref{prop:Y_basis} allows us to define the following action on $H_{\T}^*(\Gr)\otimes \R$.

%\begin{definition} \label{def:geometric_action}
%Define the action of $\sigma\in S_n$ on $H_{\T}^*(\Gr)\otimes \R$ by
%$$\sigma: e_h(\gamma^*\otimes \gamma){\kappa_I} \mapsto e_h(\gamma^*\otimes \gamma)\pi_*[C\!\tilde{S}^{\sigma}_I].$$
%\end{definition}

%\begin{proposition} We have
%The action defined in Definition \ref{def:geometric_action} is the same as the action obtained by permuting the variables $z_1,\ldots,z_n$.
%$$\sigma: e_h(\gamma^*\otimes \gamma){\kappa_I} = e_h(\gamma^*\otimes \gamma){\kappa_I} |_{z_u\to z_{\sigma(u)}}.$$
%\end{proposition}

\begin{proposition} We have
$$\kappa_I^{\sigma} = {\kappa_I} |_{z_u\mapsto z_{\sigma(u)}},$$
$${\kappa'_I}^{\sigma} = {\kappa'_I}|_{z_u\mapsto z_{\sigma(u)}}.$$
\end{proposition}

\begin{proof}
The statement follows from equivariant localization (see Section \ref{sec:localization}).
\end{proof}

\smallskip

Comparing this result with Section \ref{sec:R} we obtain the following remarkable fact. Let $\sigma$ be an elementary transposition in $S_n$. Then the geometrically defined action $\sgn(I){\kappa'_I}\mapsto \sgn(I){\kappa'_I}^{\sigma} $ on $H^*_{\T}(\Gr)\otimes \R$ can be expressed by an R-matrix.

%${\kappa_I}\mapsto \pi_*[C\!\tilde{S}_I^{\sigma}]$ on $H_{\T}^*(\Gr)\otimes \R$ can be expressed by an R-matrix.

\begin{example} \rm Let $n=2, k=1$, and let $\sigma$ be the transposition in $S_2$.
%We have $\Gr=\PP^1$,
%$$\pi_*([C\tilde{S}_{\{1\}}]=[T^*_{(1:0)}\P^1], \pi_*([C\tilde{S}_{\{2\}}]=[\P^1] \in H_{\T}^*(T^*\!\P^1).$$
Then we have
$$
\begin{pmatrix}
	\cfrac{h}{z_2-z_1+h} & \cfrac{z_2-z_1}{z_2-z_1+h} \\
	\cfrac{z_2-z_1}{z_2-z_1+h} & \cfrac{h}{z_2-z_1+h}
\end{pmatrix}
\cdot
\begin{pmatrix}
-\kappa'_{\{1\}} \\
\ \ \kappa'_{\{2\}}
\end{pmatrix}
=
\begin{pmatrix}
-{\kappa'_{\{1\}}}^\sigma \\
\ \ \-{\kappa'_{\{2\}}}^\sigma
\end{pmatrix},
$$
which can be directly verified by substituting
\begin{align*}
\sgn(\{1\})=-1, &\qquad\qquad \kappa'_{\{1\}}=\kappa_{\{1\}}=z_2-t_1, \\
\sgn(\{2\})=+1, &\qquad\qquad \kappa'_{\{2\}}=\kappa_{\{2\}}+\kappa_{\{1\}}=(2t_1-z_1-z_2+h)+(z_2-t_1)=t_1-z_1+h.
\end{align*}
%$\kappa'_{\{1\}}=-(t_1-z_2)$, $\kappa'_{\{2\}}=(t_1-z_1+h)$.
%$$R \cdot
%\begin{pmatrix}
%[W_1] \\ [W_2]
%\end{pmatrix} =
%\begin{pmatrix}
%\sigma [W_1] \\ \sigma [W_2]
%\end{pmatrix}.
%$$
%Equivalently we have
%$$
%\begin{pmatrix} -1 & 0 \\ 1 & 1 \end{pmatrix}^{-1} R \ \ \begin{pmatrix} -1 & 0 \\ 1 & 1 \end{pmatrix}
%\cdot
%\begin{pmatrix}
%\pi_*[C\!\tilde{S}_{\{1\}}] \\ \pi_*[C\!\tilde{S}_{\{2\}}]
%\end{pmatrix} =
%\begin{pmatrix}
%\pi_*[ C\tilde{S}^{\sigma}_{\{1\}}] \\ \pi_*[C\!\tilde{S}^{\sigma}_{\{2\}}],
%\end{pmatrix}.
%$$
%where $\begin{pmatrix} -1 & 0 \\ 1 & 1 \end{pmatrix}$ is the transition matrix defined by
%$$\begin{pmatrix} -1 & 0 \\ 1 & 1 \end{pmatrix} \cdot
%\begin{pmatrix}
%\pi_*[C\!\tilde{S}_{\{1\}}] \\ \pi_*[C\!\tilde{S}_{\{2\}}]
%\end{pmatrix} =
%\begin{pmatrix}
%[W_1] \\ [W_2]
%\end{pmatrix}.
%$$
\end{example}

%We may consider the fundamental cohomology classes of the geometric objects of the previous sections, with the change of reordering the Chern roots $z_1,\ldots,z_n$ of the $(\C^*)^n$ action. In Theorem \ref{thm:main} these fundamental classes are expressed in terms of weight functions. In Section~\ref{sec:R} we showed that the effect of reordering $z_i$'s on weight functions is governed by R-matrixes. Hence, the effect of reordering the %Chern roots $z_1,\ldots,z_n$ on $\pi_*[C\!S_I]$ classes is also governed by R-matrixes.

%For example for $k=1, n=2$ the relevant varieties, namely,
%$$C\!S_{\{1\}}=T_{(1,0)}^*\Proj^1 \qquad \text{and} \qquad C\!S_{\{2\}}=\Proj^1$$
%in $T^*\!\Gr_1\C^2=T^*\Proj^1$ are both smooth. We obtain that
%\begin{equation*}
%\begin{pmatrix} \cfrac{h}{z_2-z_1+h} & \cfrac{z_2-z_1}{z_2-z_1+h} \\
% \cfrac{z_2-z_1}{z_2-z_1+h} & \cfrac{h}{z_2-z_1+h} \end{pmatrix}
%\begin{pmatrix} -[C\!S_{\{1\}}] \\ [C\!S_{\{1\}}]+[C\!S_{\{2\}}] \end{pmatrix} =
%\begin{pmatrix} -[C\!S_{\{1\}}]_{z_1\leftrightarrow z_2} \\ [C\!S_{\{1\}}]_{z_1\leftrightarrow z_2}+[C\!S_{\{2\}}]_{z_1\leftrightarrow z_2} %\end{pmatrix}.
%\end{equation*}

\end{document}